\documentclass[twoside,12pt]{article}

\usepackage{amssymb, amsmath}
\usepackage{mathtools}
\usepackage[thmmarks,amsmath]{ntheorem}

\usepackage[OT1]{fontenc}

\usepackage[utf8]{inputenc}
\usepackage[T1]{fontenc}
\usepackage{dsfont}
\usepackage{latexsym} 

\usepackage{enumerate}	
\usepackage{footnpag} 
\usepackage{makeidx}  

\usepackage{geometry}	
\usepackage{parcolumns}

\usepackage{graphicx}
\usepackage{color} 

\usepackage[all]{xy}

\usepackage{url}

\usepackage[OT2,OT1]{fontenc}
\newcommand\cyr
{
\renewcommand\rmdefault{wncyr}
\renewcommand\sfdefault{wncyss}
\renewcommand\encodingdefault{OT2}
\normalfont
\selectfont
}
\DeclareTextFontCommand{\textcyr}{\cyr}

\renewcommand{\geq}{\geqslant}	
\renewcommand{\leq}{\leqslant}
\def\mid|{\hs\middle|\hs}
\def\hs{\hspace*{0.1cm}}

\def\1{\mathbb{1}}
\newcommand{\sub}{\subseteq}	
\renewcommand{\1}{\mathds{1}}					

\newcommand{\NN}{\mathbb{N}}	
\newcommand{\RR}{\mathbb{R}} 
\newcommand{\ZZ}{\mathbb{Z}} 

\def\t{\textnormal}

\newcommand{\ohne}{\backslash}

\def\c{\cite }
\newcommand{\hide}[1]{}  

\def\eps{\varepsilon}
\def\phi{\varphi}
\def\rho{\varrho}

\DeclareMathAlphabet\mathbfcal{OMS}{cmsy}{b}{n}

\newcommand{\Sup}{\vee}		
\newcommand{\Inf}{\wedge}	%

\let\int\relax 
\DeclareMathOperator{\int}{int}

\renewcommand{\theequation}{\arabic{equation}}

\usepackage{enumitem}
\setenumerate[0]{label=(\roman*)}
\setenumerate[2]{label=(\alph*)}

\newcounter{Zaehler}
\theoremstyle{plain}
\theoremheaderfont{\normalfont\bfseries}
\theorembodyfont{\itshape}
\theoremseparator{.} 
\theorempreskip{\topsep}
\theorempostskip{\topsep} 
\theoremnumbering{arabic}
\theoremsymbol{}

\newtheorem{theorem}{Theorem}
\newtheorem{lemma}[theorem]{Lemma}

\newtheorem{proposition}[theorem]{Proposition}
\newtheorem{definition}[theorem]{Definition}			

\theoremstyle{plain}
\theoremheaderfont{\normalfont\bfseries}
\theorembodyfont{\upshape}
\theoremseparator{.} 
\theorempreskip{\topsep}
\theorempostskip{\topsep} 
\theoremnumbering{arabic}
\theoremsymbol{}

\newtheorem{example}[theorem]{Example}
\newtheorem*{remark}{Remark}

\makeatletter
\newcommand{\eqnum}{\leavevmode\hfill\refstepcounter{equation}\textup{\tagform@{\theequation}}}
\makeatother

\theoremstyle{nonumberplain}
\theoremheaderfont{\itshape}
\theorembodyfont{\normalfont}
\theoremseparator{.} 
\theorempreskip{\topsep}
\theorempostskip{\topsep} 
\theoremsymbol{\ensuremath\square} 
\qedsymbol{$\Box$}

\newtheorem{proof}{Proof}

\begin{document}
\title{Pervasive and weakly pervasive pre-Riesz spaces}
\author{Anke Kalauch\footnote{FR Mathematik, Institut für Analysis, TU Dresden, 01062 Dresden, Germany,\newline \texttt{anke.kalauch@tu-dresden.de}},
Helena Malinowski\footnote{Unit for BMI, North-West University, Private Bag X6001, Potchefstroom, 2520, South Africa,\newline \texttt{lenamalinowski@gmx.de}}}

\maketitle
\begin{abstract}
Pervasive pre-Riesz spaces are defined by means of vector lattice covers. To avoid the computation of a vector lattice cover, we give two distinct intrinsic characterizations of pervasive pre-Riesz spaces. We introduce weakly pervasive pre-Riesz spaces and observe that this property can be easily checked in examples. We relate weakly pervasive pre-Riesz spaces to pre-Riesz spaces with the Riesz decomposition property.
\end{abstract}




\section{Introduction}
In the analysis of ordered vector spaces which are not vector lattices there are certain additional properties, such as pervasiveness, fordability and the Riesz decomposition property, which allow to generalize well-known results from the vector lattice theory. In the present paper we characterize and relate some of these properties in pre-Riesz spaces.

An ordered vector space $X$ is called pre-Riesz if there is a vector lattice $Y$
and a bipositive linear map $i\colon X \rightarrow Y$ such that $i(X)$ is order dense in $Y$. The pair $(Y,i)$ is then called a vector lattice cover of $X$. The theory of pre-Riesz spaces and
their vector lattice covers is due to van Haandel, see \c{vanHaa}. Pre-Riesz spaces cover a
wide range of examples, in particular every Archimedean directed ordered
vector space is a pre-Riesz space.

We mainly deal with pervasive pre-Riesz spaces, i.e.\ spaces $X$ such that for every $y\in Y$ with $y>0$ there is $x\in X$ with $0<i(x)\leq y$.
Pervasive pre-Riesz spaces were introduced in \c{3}.
To illustrate their importance, we list some results from the literature where pervasive pre-Riesz spaces play an essential role.
In \c[Theorem~2.6]{3} pervasiveness is a crucial assumption to show that the restriction of a band in $Y$ to $X$ is a band in $X$. Moreover, it is established that the space $L_r(\ell_0^\infty)$ of regular operators on the space $\ell_0^\infty$ of eventually constant sequences is a pervasive pre-Riesz space. This yields that the space $L_{oc}(\ell_0^\infty)$ of order continuous operators is a band in $L_r(\ell_0^\infty)$. This statement is a first instance of an Ogasawara type result where the range space is not Dedekind complete. 
In \c[Example 22]{04} the authors prove that $L_{oc}(\ell_0^\infty)$ is pervasive.
In \c[Theorem 14]{02} it is shown that in a pervasive pre-Riesz space with the Riesz decomposition property every directed order closed ideal with a directed double disjoint complement is a band. 
In \c[Theorem 27]{03} it is established that for a directed ideal $I$ in $X$ a vector lattice cover of $I$ is given by the smallest extension ideal of $I$ in $Y$, provided $X$ is pervasive.
In \c[Theorem 5.3]{vanImhoff2017} the author obtains that the inverse of a bijective Riesz* homomorphism between pre-Riesz spaces is again a Riesz* homomorphism, provided the pre-image space is pervasive. In \c[Theorem 7.2]{vanImhoff2017} spaces of differentiable functions which are defined on sufficiently smooth manifolds and vanish at infinity are shown to be pervasive pre-Riesz spaces.

So far, the definition of pervasiveness is given using a vector lattice cover. In examples it might be difficult to find a convenient representation of a vector lattice cover of a given pre-Riesz space. Therefore in Section~\ref{intrinsic} we give two intrinsic characterizations of pervasiveness, which do not use vector lattice covers.
In Section~\ref{weakly_pervasive} we introduce weakly pervasive pre-Riesz spaces and give a characterization, such that this property can be easily checked in examples. 
We relate pervasive and weakly pervasive pre-Riesz spaces to pre-Riesz spaces with the Riesz decomposition property.
The following diagram illustrates the results and counterexamples obtained in Section~\ref{weakly_pervasive}.
\begin{center}
\begin{picture}(0,0)%
\includegraphics{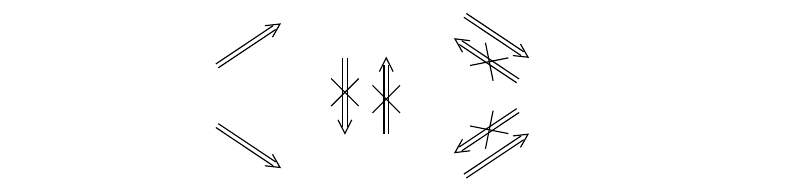}%
\end{picture}%
\setlength{\unitlength}{2901sp}%
\begingroup\makeatletter\ifx\SetFigFont\undefined%
\gdef\SetFigFont#1#2#3#4#5{%
  \reset@font\fontsize{#1}{#2pt}%
  \fontfamily{#3}\fontseries{#4}\fontshape{#5}%
  \selectfont}%
\fi\endgroup%
\begin{picture}(5138,1185)(-1216,-1516)
\put( 91,-961){\makebox(0,0)[rb]{\smash{{\SetFigFont{8}{9.6}{\familydefault}{\mddefault}{\updefault}{\color[rgb]{0,0,0}vector lattice}%
}}}}
\put(1126,-466){\makebox(0,0)[b]{\smash{{\SetFigFont{8}{9.6}{\familydefault}{\mddefault}{\updefault}{\color[rgb]{0,0,0}pervasive}%
}}}}
\put(1216,-1501){\makebox(0,0)[b]{\smash{{\SetFigFont{8}{9.6}{\familydefault}{\mddefault}{\updefault}{\color[rgb]{0,0,0}RDP}%
}}}}
\put(2431,-961){\makebox(0,0)[lb]{\smash{{\SetFigFont{8}{9.6}{\familydefault}{\mddefault}{\updefault}{\color[rgb]{0,0,0}weakly pervasive}%
}}}}
\end{picture}%

\end{center}
Since we characterize the weak pervasiveness intrinsically (i.e.\ without using vector lattice covers), our results provide a convenient method to show that a given pre-Riesz space is not pervasive. This is used to answer the open question whether the space in \cite{Gaa2005} is pervasive, see Example~\ref{properties.14} below.

\section{Preliminaries}
Let $X$ be a real vector space and let $X_+$ be a \emph{cone} in $X$, that is,  $X_+$ is a wedge ($x,y\in X_+$ and $\lambda,\mu\geq 0$ imply $\lambda x + \mu y \in X_+$) and $X_+ \cap (-X_+) =\left\{0\right\}$. In $X$ a partial order is defined by $x\leq y$ whenever $y-x\in X_+$. 
The space $(X,X_+)$ (or, loosely $X$) is then called a (\emph{partially}) \emph{ordered vector space}.

An ordered vector space $X$ is called \emph{Archimedean} if for every $x,y\in X$ with $nx\leq y$ for every $n\in\NN$ one has $x\leq 0$. 
Clearly, every subspace of an Archimedean ordered vector space is Archimedean. The ordered vector space $X$ is called \emph{directed} if for every $x,y\in X$ there is $z\in X$ such that $x,y\leq z$. The space $X$ is directed if and only if  $X_+$ is \emph{generating} in $X$, that is, $X = X_+ - X_+$. 
An ordered vector space $X$ has the \emph{Riesz decomposition property} (\emph{RDP}) if for every $x_1,x_2,z \in X_+$ with $z \leq x_1+x_2$ there exist $z_1,z_2 \in X_+$ such that $z = z_1+z_2$ with $z_1 \leq x_1$ and $z_2\leq x_2$. The space $X$ has the RDP if and only if for every $x_1,x_2,x_3,x_4\in X$ with $x_1,x_2 \leq x_3,x_4$ there exists $z\in X$ such that $x_1, x_2 \leq z \leq x_3, x_4$.
For standard notations in the case that $X$ is a vector lattice see \c{PosOp}.

For two elements $a,b\in X$ of an ordered vector space $X$ we use the notation $]a,b]:=\left\{x\in X\mid| a<x\leq b\right\}$.
For a set $A\sub X$ we set $A+x:=\left\{a+x \mid| a\in A\right\}$. Moreover, we write $A\leq x$ if and only if for every $a\in A$ we have $a\leq x$.
By \c[Theorem~13.1]{Zaa1} for a set $A\sub X$ such that $\sup A$ exists in $X$ and an element $x\in X$ we have that the supremum $\sup(x + A)$ exists in $X$ and
\begin{equation}\label{Supremum_reinziehen}
x+ \sup A = \sup(x + A).
\end{equation}
Clearly, we can replace the supremum by the infimum.

We call a linear subspace $D$ of an ordered vector space $X$ \emph{order dense} in $X$ if for every $x\in X$ we have \[x = \inf\left\{z\in D \mid| x\leq z\right\},\] see \c[p.~360]{113}. A linear subspace $D$ of $X$ is \emph{majorizing} in $X$ if for every $x\in X$ there exists $d\in D$ with $x\leq d$.
For a subset $M\sub X$ denote the set of all upper bounds of $M$ by $M^u:=\left\{x\in X\mid| \forall m\in M\colon m\leq x\right\}$. Two elements $x,y\in X$ are called \emph{disjoint}, in symbols $x\perp y$, if $\left\{x+y,-x-y\right\}^u = \left\{x-y,-x+y\right\}^u$, for motivation and details see \c{1}. For a subset $M\sub X$ define $M^{\t{d}}:=\left\{x\in X\mid| \forall m\in M\colon m\perp x\right\}$. If $X$ is a vector lattice, then this notion of disjointness coincides with the usual one, see \c[Theorem~1.4(4)]{PosOp}. Let $Y$ be an ordered vector space, $X$ an order dense subspace of $Y$, and $x,y\in X$. Then the disjointness notions in $X$ and $Y$ coincide, i.e.\ $x\perp y$ in $X$ holds if and only if $x\perp y$ in $Y$, see \c[Proposition~2.1(ii)]{1}.

We say that a linear subspace $D$ of a vector lattice $X$ \emph{generates $X$ as a vector lattice} if for every $x\in X$ there exist finite sets $A, B\sub D$ such that $x = \bigvee A - \bigvee B$. 
Recall that a linear map $i\colon X\rightarrow Y$, where $X$ and $Y$ are ordered vector spaces, is called \emph{bipositive} if for every $x\in X$ one has $i(x) \geq 0$ if and only if $x\geq 0$. An \emph{embedding} is a bipositive linear map, which implies injectivity. 
For an ordered vector space $X$, the following statements are equivalent, see \c[Corollaries~4.9-4.11 and Theorems~3.7, 4.13]{vanHaa}:
\begin{enumerate}
\item\label{embedding.it1} There exist a vector lattice $Y$ and an embedding $i\colon X\rightarrow Y$ such that $i(X)$ is order dense in $Y$.
\item\label{embedding.it2} There exist a vector lattice $\tilde{Y}$ and an embedding $i\colon X\rightarrow \tilde{Y}$ such that $i(X)$ is order dense in $\tilde{Y}$ and generates $\tilde{Y}$ as a vector lattice.
\end{enumerate}
If $X$ satisfies \ref{embedding.it1}, then $X$ is called a \emph{pre-Riesz space}, and $(Y,i)$ is called a \emph{vector lattice cover} of $X$. For an intrinsic definition of pre-Riesz spaces see \c{vanHaa}. If $X$ is a subspace of $Y$ and $i$ is the inclusion map, we write briefly $Y$ for $(Y,i)$.
As all spaces $\tilde{Y}$ in \ref{embedding.it2} are Riesz isomorphic, we call the pair $(\tilde{Y},i)$ the \emph{Riesz completion of} $X$ and denote it by $X^\rho$. 
The space $X^\rho$ is the smallest vector lattice cover of $X$ in the sense that every vector lattice cover $Y$ of $X$ contains a Riesz subspace that is Riesz isomorphic to $X^\rho$.
By definition, for every $y\in X^\rho$ there are finite sets $A, B\sub X$ such that
\begin{equation}\label{representation}
y= \bigvee i(A) - \bigvee i(B).
\end{equation}
\begin{lemma}\label{properties.2}
Let $X$ be a pre-Riesz space and $X^\rho$ its Riesz completion. Let $y\in X^\rho$. Then in \eqref{representation} the sets $A$ and $B$ can be chosen to be contained in $X_+$.
\end{lemma}
\begin{proof}
Let $y\in X^\rho$ have a representation as in \eqref{representation}, i.e.\ $y= \bigvee i(\tilde{A}) - \bigvee i(\tilde{B})$ with $\tilde{A}, \tilde{B}\sub X$.
Since the pre-Riesz space $X$ is directed there exists $x\in X$ with $0\leq x$, $\tilde{A}\leq x$ and $\tilde{B}\leq x$. Then
\begin{align*}
y &= \bigvee i(\tilde{A}) - \bigvee i(\tilde{B})
= \left(i(x) - \bigvee i(\tilde{B})\right) - \left(i(x) - \bigvee i(\tilde{A})\right) =\\
&= \bigvee i(x-\tilde{B}) - \bigvee i(x-\tilde{A}),
\end{align*}
where the last equality follows by \eqref{Supremum_reinziehen}. 
Clearly, for $A:= x-\tilde{B}$ and $B:=x-\tilde{A}$ we have $A,B\sub X_+$.
\end{proof}

By \c[Theorem 17.1]{vanHaa} every Archimedean directed ordered vector space is a pre-Riesz space. Moreover, every pre-Riesz space is directed.
If $X$ is an Archimedean directed ordered vector space, then every vector lattice cover of $X$ is Archimedean.  

Let $X$ be a pre-Riesz space and $(Y,i)$ a vector lattice cover of $X$.
The space $X$ is called \emph{pervasive in} $Y$ if for every $y\in Y_+$, $y\neq 0$, there exists $x\in X$ such that $0<i(x) \leq y$. By \c[Proposition~2.8.8]{PRS} the space $X$ is pervasive in $Y$ if and only if $X$ is pervasive in any vector lattice cover. Then $X$ is simply called \emph{pervasive}.
The space $X$ is called \emph{fordable in} $Y$ if for every $y\in Y$ there exists a set $S\sub X$ such that $\left\{y\right\}^{\t{d}} = i(S)^{\t{d}}$ in $Y$. By \c[Proposition~4.1.18]{PRS} the space $X$ is fordable in $Y$ if and only if $X$ is fordable in any vector lattice cover of $X$. Then $X$ is simply called \emph{fordable}. By \c[Lemma~2.4]{3} every pervasive pre-Riesz space is fordable.

\section{Intrinsic characterizations of pervasiveness}\label{intrinsic}
In a pre-Riesz space, the definition of pervasiveness depends on the Riesz completion or a vector lattice cover. An intrinsic definition allows to check whether a space is pervasive without having to compute a vector lattice cover. In Theorems~\ref{properties.5} and \ref{properties.2b} we give two distinct intrinsic characterizations.
We start with a recollection of two (non-intrinsic) characterizations of pervasiveness, which we include for the sake of completeness. 
For the following result from \cite[Theorem~4.15, Corollary~4.16]{Waaij} a short proof can also be found in \c[Lemma~1]{02}.
\begin{proposition}\label{0.0}
Let $X$ be an Archimedean pre-Riesz space and $(Y,i)$ a vector lattice cover of $X$. Then the following statements are equivalent.
\begin{enumerate}
\item\label{properties.1.it1} $X$ is pervasive.
\item\label{properties.1.it2} $\forall a\in X\hs\forall y\in Y\hs \big(i(a)<y \Rightarrow \exists x\in X\colon i(a)<i(x)\leq y\big)$.
\item\label{properties.1.it3} For every $y\in Y_+$ with $y\neq 0$ we have $y = \sup\left(i(X)\cap\left]0, y\right]\right)$.
\item\label{properties.1.it4} For every $y\in Y$ and $z\in X$ with $i(z)< y$ we have
$y = \sup\left(i(X)\cap\left]i(z),y\right]\right)$.
\end{enumerate}
\end{proposition}

The following result was established in \c[Theorem~26]{03}.
\begin{proposition}\label{properties.28neu}
Let $X$ be an Archimedean pre-Riesz space, $Y$ a vector lattice and $i\colon X\rightarrow Y$ a bipositive linear map. Then the following statements are equivalent.
\begin{enumerate}
\item\label{properties.28neu.it1} $i(X)$ is order dense in $Y$ and $X$ is pervasive.
\item\label{properties.28neu.it2} $i(X)$ is majorizing in $Y$, and for every $y\in Y_+$ there is $A\sub i(X)_+$ such that $ y=\sup A$.
\end{enumerate}
\end{proposition}

For our first intrinsic characterization of pervasiveness we begin with the following technical statement.
\begin{lemma}\label{properties.3}
Let $X$ be a pre-Riesz space and $y\in X^\rho$. Then
\begin{enumerate}
\item\label{properties.3.it1} $y\geq 0$ if and only if for every two finite sets $A,B\sub X_+$ with $y=\bigvee i(A) - \bigvee i(B)$ we have
$A^u\sub B^u$.
\item\label{properties.3.it2} $y>0$ if and only if for every two finite sets $A,B\sub X_+$ with $y=\bigvee i(A) - \bigvee i(B)$ we have
$A^u\subsetneq B^u$.
\end{enumerate}
\end{lemma}
\begin{proof}
\ref{properties.3.it1}: Let $y\in X^\rho$ with $y\geq 0$ and let $A,B\sub X_+$ be finite sets such that $y=\bigvee i(A) - \bigvee i(B)$. The existence of such sets is guaranteed by Lemma~\ref{properties.2}. Then we have $\bigvee i(B) \leq \bigvee i(A)$. If $v\in X^\rho$ is an upper bound of $i(A)$, then it is an upper bound of $i(B)$, i.e.\ $i(A)^u\sub i(B)^u$. It follows $A^u=[i(A)^u]i \sub [i(B)^u]i = B^u$.

Let, on the other hand, $y\in X^\rho$ and let $A,B\sub X_+$ be finite sets with $y=\bigvee i(A) - \bigvee i(B)$ and $A^u\sub B^u$. We show $i(A)^u\sub i(B)^u$. Let $v\in i(A)^u$. Due to $X$ being order dense in $X^\rho$, for the set $V:=\left\{i(x)\mid|x\in X,\hs v\leq i(x)\right\}\sub i(X)$ we have $v=\inf V$ in $X^\rho$. For every $x\in X$ with $i(x)\in V$ from $i(A)\leq v\leq i(x)$  we obtain $A\leq x$, i.e.\ $x\in A^u\sub B^u$. It follows $i(x)\in i(B)^u$, i.e.\ $i(B)\leq i(x)$. As this is true for every $x\in X$ with $i(x)\in V$, it follows $i(B)\leq \inf V =v$. That is, $v\in i(B)^u$. We obtain $i(A)^u\sub i(B)^u$, from which we conclude $\bigvee i(B) \leq \bigvee i(A)$ and thus $y\geq 0$.

\ref{properties.3.it2}: Let $y\in X^\rho$ with $y> 0$ and let $A,B\sub X_+$ be finite sets such that $y=\bigvee i(A) - \bigvee i(B)$. Then $y\geq 0$, thus \ref{properties.3.it1} implies $A^u= B^u$. We show $A^u\neq B^u$ by contradiction. Assume, on the contrary, that $A^u= B^u$. Due to $y>0$ we have $t:=\bigvee i(B) < \bigvee i(A)=:s$. Since $X$ is order dense in $X^\rho$ it follows
\begin{align*}
t=&\inf\left\{i(x)\mid| x\in X,\hs t\leq i(x)\right\}=\inf \left\{i(x)\mid| x\in X,\hs A\leq x\right\}=\inf i(A^u)=\\
=&\inf i(B^u)=\inf\left\{i(x)\mid| x\in X,\hs s\leq i(x)\right\}=s,
\end{align*}
a contradiction. We conclude that $A^u\subsetneq B^u$.

Let, on the other hand, $y\in X^\rho$ and let $A,B\sub X_+$ be finite sets with $y=\bigvee i(A) - \bigvee i(B)$ and $A^u\subsetneq B^u$. Then $A^u\sub B^u$ and \ref{properties.3.it1} imply $y\geq 0$. We show $y>0$ by contradiction. Let, on the contrary, $y=0$. Then $t:=\bigvee i(B) =\bigvee i(A)=:s$. For a fixed $x\in X$ the statement $A\leq x$ is equivalent to $\bigvee i(A)\leq i(x)$, and similarly, $B\leq x$ is equivalent to $\bigvee i(B)\leq i(x)$. It follows
\begin{align*}
A^u=&\left\{x\in X\mid| A\leq x\right\}=\left\{x\in X\mid| t\leq i(x)\right\}=\left\{x\in X\mid| s\leq i(x)\right\}=\\
=&\left\{x\in X\mid| B\leq x\right\}=B^u,
\end{align*}
a contradiction to $A^u\subsetneq B^u$. We conclude that $y>0$.
\end{proof}

\begin{theorem}\label{properties.5}
Let $X$ be a pre-Riesz space. Then $X$ is pervasive if and only if for every two finite sets $A, B \sub X_+$ with $A^u\subsetneq B^u$ there exists $x\in X$, $x> 0$, such that $A^u\sub (x+B)^u$.
\end{theorem}
\begin{proof}
``$\Rightarrow$'': Let $X$ be pervasive and let $(X^\rho,i)$ be the Riesz completion of $X$. Let $A,B\sub X_+$ be finite sets with $A^u\subsetneq B^u$. Set $y:=\bigvee i(A) - \bigvee i(B)$. Then by Lemma~\ref{properties.3}\ref{properties.3.it2} it follows $y>0$. As $X$ is pervasive, there exists an $x\in X$ with $0<i(x)\leq y$. Hence
\[0\leq y -i(x) = \bigvee i(A) - \left(\bigvee i(B) +i(x)\right) = \bigvee i(A) - \left(\bigvee i(B+x)\right).\]
Clearly, $B+x \sub X_+$, therefore by Lemma~\ref{properties.3}\ref{properties.3.it1} we obtain $A^u\sub (x+B)^u$.

``$\Leftarrow$'': Let for every two finite sets $A, B \sub X_+$ with $A^u\subsetneq B^u$ there exist $x\in X_+$, $x\neq 0$, such that $A^u\sub (x+B)^u$. Let $y\in X^\rho$ with $y>0$. Then by Lemma~\ref{properties.2} there exist finite sets $A, B\sub X_+$ such that $y$ has the representation
$y = \bigvee i(A) - \bigvee i(B)$. As $y>0$, by Lemma~\ref{properties.3}\ref{properties.3.it2} it follows $A^u\subsetneq B^u$. By assumption there exists $x\in X$, $x>0$, such that $A^u\sub (x+B)^u$. Lemma~\ref{properties.3}\ref{properties.3.it1} yields $0\leq y-i(x)$, i.e.\ $X$ is pervasive.
\end{proof}

Next we intend to give a characterization of pervasiveness using pairs of elements instead of finite subsets.
\begin{proposition}\label{properties.6x}
Let $X$ be a pre-Riesz space and $(Y,i)$ a vector lattice cover of $X$. Then the following are equivalent.
\begin{enumerate}
	\item\label{properties.6x.it1} $X$ is pervasive.
	\item\label{properties.6x.it2} For every $b_1, b_2\in X$ with $i(b_1)\Sup i(b_2)> 0$ there exists $x\in X$ such that $0 < i(x) \leq i(b_1)\Sup i(b_2)$.
	\item\label{properties.6x.it3} For every $b\in X$ with $i(b)\Sup 0> 0$ there exists $x\in X$ such that $0 < i(x) \leq i(b)\Sup 0$.
\end{enumerate}
\end{proposition}
\begin{proof}
The implications \ref{properties.6x.it1}$\Rightarrow$\ref{properties.6x.it2} and \ref{properties.6x.it2}$\Rightarrow$\ref{properties.6x.it3} are clear.

%
\ref{properties.6x.it3}$\Rightarrow$\ref{properties.6x.it1}: Let $y\in Y$ with $y> 0$. Since $X$ is order dense in $Y$, we have $y=\sup\left\{z\in i(X)\mid| z\leq y\right\}$.
Thus there exists a lower bound $z\in i(X)$ of $y$ for which we have $z\not\leq 0$. For $b:=i^{-1}(z)$ it follows $0<i(b)\Sup 0$. By assumption there exists $x\in X$ such that $0<i(x)\leq i(b)\Sup 0$. It follows $0<i(x)\leq i(b)\Sup 0 \leq y$, i.e.\ $X$ is pervasive.
\end{proof}

By reformulating Proposition~\ref{properties.6x} with the help of upper bounds we obtain the following intrinsic characterization of pervasiveness. This characterization is used in Example~\ref{properties.7} below, where we establish that a pre-Riesz space is pervasive without prior computation of its vector lattice cover.
\begin{theorem}\label{properties.2b}
Let $X$ be an Archimedean pre-Riesz space. Then the following are equivalent.
\begin{enumerate}
\item\label{properties.2b.it1} $X$ is pervasive.
\item\label{properties.2b.it2} For every $b_1, b_2\in X$ with $\left\{b_1,b_2\right\}^u\sub X_+\ohne\left\{0\right\}$ there is $x\in X$ such that for every $u\in\left\{b_1,b_2\right\}^u$ we have $0 < x \leq u$.
\item\label{properties.2b.it3} For every $b\in X$ with $b\not\leq 0$ there is $x\in X$ such that for every $u\in X_+$ the inequality $b\leq u$ implies $0 < x \leq u$.
\end{enumerate}
\end{theorem}
\begin{proof}
Let $(Y,i)$ be a vector lattice cover of $X$.

\ref{properties.2b.it1}$\Rightarrow$\ref{properties.2b.it2}:
Let $b_1,b_2\in X$ be two elements such that $0<u'$ holds for every upper bound $u'\in\left\{b_1,b_2\right\}^u$.
Since $X$ is order dense in $Y$, this implies
\begin{align*}
i(b_1)\Sup i(b_2) &= \inf\left\{z\in i(X)\mid| i(b_1)\Sup i(b_2)\leq z\right\} = \\
&=\inf\left\{i(u') \mid| u'\in\left\{b_1,b_2\right\}^u\right\} \geq 0.
\end{align*}
Now, if $i(b_1)\Sup i(b_2) =0$, then the element $0$ is an upper bound of $b_1,b_2$. This is a contradiction to the assumption that $u'>0$ for every upper bound $u'\in\left\{b_1,b_2\right\}^u$.
Therefore we have $i(b_1)\Sup i(b_2)>0$. Since $X$ is pervasive, there exists $x\in X$ with $0<i(x)\leq i(b_1)\Sup i(b_2)$. This implies $0<x\leq u$ for every upper bound $u\in\left\{b_1,b_2\right\}^u$.

\ref{properties.2b.it2}$\Rightarrow$\ref{properties.2b.it1}: We  use the characterization in Proposition~\ref{properties.6x}\ref{properties.6x.it2}. Let $b_1,b_2\in X$ be such that $i(b_1)\Sup i(b_2)>0$. We have to show that there exists an element $x\in X$ with $0<i(x)\leq i(b_1)\Sup i(b_2)$. Let $u'\in\left\{b_1,b_2\right\}^u$. Then $b_1,b_2\leq u'$ yields $0< i(b_1)\Sup i(b_2)\leq i(u')$, which implies $0<u'$. By assumption there exists $x\in X$ such that for every $u\in\left\{b_1,b_2\right\}^u$ we have $0<x\leq u$. Due to $X$ being order dense in $Y$ we obtain in $Y$ the inequality
\[0<i(x)\leq \inf\left\{i(u)\mid| u\in X,\hs i(b_1)\Sup i(b_2)\leq i(u)\right\} = i(b_1)\Sup i(b_2).\]
The characterization in  Proposition~\ref{properties.6x}\ref{properties.6x.it2} implies that $X$ is pervasive.

\ref{properties.2b.it1}$\Rightarrow$\ref{properties.2b.it3}: Let $b\in X$ with $b\not\leq 0$. We show
\[\exists x\in X \hs\forall u\in X_+\colon\hs b\leq u\Rightarrow 0<x\leq u.\]
We have $i(b)\Sup 0\geq 0$. If $i(b)\Sup 0 = 0$, then $b\leq 0$, a contradiction. We conclude that $i(b)\Sup 0 >0$. 
Since $X$ is pervasive, we obtain that there exists an $x\in X$ with $0<i(x)\leq i(b)\Sup 0$. That is, for every $u\in X$ with $0,b\leq u$ we have $0<i(x)\leq i(b)\Sup 0 \leq i(u)$. This leads to the inequality $0<x\leq u$.

\ref{properties.2b.it3}$\Rightarrow$\ref{properties.2b.it1}: We use the characterization in Proposition~\ref{properties.6x}\ref{properties.6x.it3}. Let $b\in X$ be such that $i(b)\Sup 0>0$. We have to show that there exists $x\in X$ with $0<i(x)\leq i(b)\Sup 0$.

If $b<0$, then $i(b)\Sup 0\leq 0$, a contradiction. 
We conclude $b\not\leq 0$. By assumption there exists $x\in X$ such that for every $u\in X_+$ with $b\leq u$ we have $0<x\leq u$. Due to $X$ being order dense in $Y$ there exists the infimum in the following inequality:
\begin{align*}
0 &< i(x) \leq \inf\left\{i(u)\mid| u\in X_+ \t{ and } b\leq u\right\} =\\
&= \inf\left\{i(u)\mid|  u\in X \t{ and } 0\Sup i(b)\leq i(u)\right\} = 0\Sup i(b).
\end{align*}
The characterization in Proposition~\ref{properties.6x}\ref{properties.6x.it3} implies that $X$ is pervasive.
\end{proof}

\section{Weakly pervasive pre-Riesz spaces}\label{weakly_pervasive}
In this section we relate pervasiveness to other properties of pre-Riesz spaces.
In Proposition~\ref{properties.6x}\ref{properties.6x.it2} pervasiveness of a pre-Riesz space $X$ is characterized with the aid of the supremum of two elements of $i(X)$. 
The question arises whether the supremum can be replaced by the infimum, which yields the following definition.

\begin{definition}
Let $X$ be a pre-Riesz space and $(Y,i)$ a vector lattice cover of $X$. Then $X$ is called \textbf{weakly pervasive} (in $Y$) if
for every $b_1, b_2\in X$ with $i(b_1)\Inf i(b_2)> 0$ there exists $x\in X$ such that $0 < i(x) \leq i(b_1)\Inf i(b_2)$.
\end{definition}
It is clear that every pervasive pre-Riesz space is weakly pervasive. The converse is not true, see Example~\ref{properties.17} below. 
We first characterize weakly pervasive pre-Riesz spaces.
\begin{lemma}\label{properties.15}
Let $X$ be a pre-Riesz space and $(Y, i)$ a vector lattice cover of $X$. Then the following statements are equivalent.
\begin{enumerate}
\item\label{properties.15.eq0} $X$ is weakly pervasive in $Y$.
\item\label{properties.15.eq2} For every $a, b_1, b_2\in X$ with $i(b_1)\Inf i(b_2)> i(a)$ there exists an element $x\in X$ such that $a < x \leq b_1, b_2$.
\item\label{properties.15.eq1} For every $b_1, b_2\in X$ with $b_1,b_2>0$ and $b_1\not\perp b_2$ there exists an element $x\in X$ such that $0 < x \leq b_1, b_2$.
\end{enumerate}
\end{lemma}
\begin{proof}
The equivalence \ref{properties.15.eq0} $\Leftrightarrow$ \ref{properties.15.eq2} and
the implication \ref{properties.15.eq0} $\Rightarrow$ \ref{properties.15.eq1} are immediate.
To show the implication \ref{properties.15.eq1} $\Rightarrow$ \ref{properties.15.eq0}, let $b_1, b_2\in X$ with $i(b_1)\Inf i(b_2)>0$. Then it follows $b_1, b_2>0$ and $b_1\not\perp b_2$. By assumption there exists $x\in X$ with $0<x\leq b_1,\hs b_2$, i.e. we have $0\leq i(x) \leq i(b_1)\Inf i(b_2)$.
\end{proof}
Note that the statement in Lemma~\ref{properties.15}\ref{properties.15.eq1} does not depend on the vector lattice cover. Similarly to pervasiveness, $X$ is weakly pervasive in $X^\rho$ if and only if $X$ is weakly pervasive in any vector lattice cover of $X$. This justifies that we simply call $X$ weakly pervasive.

As every pervasive pre-Riesz space is fordable, the next example establishes that a weakly pervasive pre-Riesz space need not be pervasive.
\begin{example}\label{properties.17}
\textit{An Archimedean weakly pervasive pre-Riesz space need not be fordable.}

Consider the vector space
\[Y=\left\{(y_i)_{i\in\ZZ} \in\ell^\infty(\ZZ)\mid| \lim_{i\to\infty}y_i \t{ exists}\right\}.\]
Endowed with the coordinatewise order, $Y$ is an Archimedean vector lattice.
In \c[Example~5.2]{2} it is shown that the Archimedean directed subspace
\[X:=\left\{(x_i)_{i\in\ZZ}\in\ell^\infty(\ZZ) \mid| \sum_{k=1}^\infty \frac{x_{-k}}{2^k}=\lim_{i\to\infty}x_i\right\}\]
of $Y$ is order dense in $Y$, that is, $X$ is an Archimedean pre-Riesz space and $Y$ is a vector lattice cover of $X$.

We show by contradiction that $X$ is not fordable. Let $z\in \ell^\infty(\ZZ)$ be such that $z_{-1}=1$ and $z_k=0$ for every $k\neq -1$, i.e.\ $z=(\ldots,0,0,0,1,z_0=0,0,0,\ldots)\in Y$. Then we have
\[\left\{z\right\}^{\t{d}}= \left\{y=(y_k)_{k\in\ZZ}\in Y\mid| y_{-1}=0\right\}.\]
Assume that $X$ is fordable and let $S\sub i(X)$ be such that $S^{\t{d}}=\left\{z\right\}^{\t{d}}$. Let $x\in Y$ be defined by $x_{-1}=0$ and $x_k=1$ for $k\neq -1$, that is, $x=(\ldots,1,1,0,x_0=1,1,\ldots)$. Then $x\in \left\{z\right\}^{\t{d}}=S^{\t{d}}$. Since for every $s\in S$ we have $s\perp x$, it follows $s=(\ldots,0,0,s_{-1},s_0=0,0,0,\ldots)$, where $s_{-1}\in\RR$. Clearly, $\lim_{k\to\infty}s_k=0$. Since $s$ belongs to $i(X)$, we obtain
\[0=\lim_{k\to\infty}s_k=\sum_{k=1}^\infty \frac{s_{-k}}{2^k}=\frac{s_{-1}}{2}.\]
This leads to $s_{-1}=0$ and therefore $s=0$. It follows $S=\left\{0\right\}$. This implies $S^{\t{d}} = Y \neq \left\{z\right\}^{\t{d}}$, a contradiction. We conclude that the pre-Riesz space $X$ is not fordable.

It is left to show that $X$ is weakly pervasive. Let $b,c\in X_+$ be two non-zero elements with $b\not\perp c$. Due to $i(b)\Inf i(c)\neq 0$ and $b,c>0$ there exists a coordinate $j$ such that $b_j,c_j>0$. Consider the following two cases.

For the first case, let $j\geq 0$. Define the sequence $x=(x_k)_{k\in\ZZ}$ by $x_j:=b_j\Inf c_j>0$ and $x_k:=0$ for $k\neq j$. Then we have
$\sum_{k=1}^\infty\frac{x_{-k}}{2^k} = 0 =\lim_{k\to\infty}x_k$
and thus $x\in X$ and $0<x\leq b, c$.

For the second case, let $j<0$. Set $C:=b_j\Inf c_j>0$. Since all coordinates of $b$ and $c$ are non-negative, we obtain the following estimates:
\[\sum_{k=1}^\infty\frac{b_{-k}}{2^k} \hs\geq\hs \frac{b_j} {2^{-j}} \hs\geq\hs \frac{C}{2^{-j}}>0
\quad\t{ and }\quad \sum_{k=1}^\infty\frac{c_{-k}}{2^k} \hs\geq\hs \frac{c_j} {2^{-j}} \hs\geq\hs \frac{C}{2^{-j}}>0.\]
Due to $b,c\in X$ this leads to
$\lim_{k\to\infty}b_k\geq \frac{C}{2^{-j}}>0$ and $\lim_{k\to\infty}c_k\geq \frac{C}{2^{-j}}>0$. 
Therefore there exist $N\in\NN_0$ and $\eps\in\RR_{>0}$ such that
\begin{equation}\label{properties.17.eq1}
\t{for every } k>N \t{ we have } b_k,c_k\geq \eps.
\end{equation}
Let $\alpha_1,\alpha_2\in\RR$ be defined by
$\alpha_1:=\min\left\{\eps,C\right\} $ and $\alpha_2:=\frac{\alpha_1}{2^{-j}}$.
Due to $j<0$ we obtain $\alpha_2=\frac{\min\left\{\eps,C\right\}}{2^{-j}}\leq \frac{\eps}{2^{-j}} <\eps$. Define a sequence $x=(x_k)_{k\in\ZZ}$ as follows:
\[x_k:=\begin{cases}
		\alpha_1 & \t{for }k=j,\\
		\alpha_2 & \t{for }k>N,\\
		0 & \t{ otherwise.}
		\end{cases}\]
Then 
$\lim_{k\to\infty}x_k = \alpha_2 <\eps$ and $\sum_{k=1}^\infty\frac{x_{-k}}{2^k} = \frac{\alpha_1}{2^{-j}}=\alpha_2$,
which yields $x\in X$.
On the one hand, from $\eps, C>0$ it follows $\alpha_1,\alpha_2>0$ and therefore $x>0$.
On the other hand, we have $x\leq b,c$. Indeed, for $k=j$ we obtain
\[x_j=\alpha_1=\min\left\{\eps,C\right\} \leq C = b_j\Inf c_j\]
and for $k>N$, using \eqref{properties.17.eq1} in the last step of the following equation, we have
\[x_k =\alpha_2 =\min\left\{\frac{\eps}{2^{-j}},\frac{C}{2^{-j}}\right\} \leq \eps \leq b_k, c_k.\]
We conclude that $0<x\leq b,c$, i.e.\ $X$ is weakly pervasive.
\end{example}
It remains an open question whether every fordable pre-Riesz space is weakly pervasive\footnote{In \c[Example~4.1.19]{PRS} the authors considers a pre-Riesz space $X$ which is fordable, but not pervasive. One can show that $X$ is weakly pervasive.}.
Another important property of pre-Riesz spaces is the RDP, which we next relate to the weak pervasiveness.
\begin{proposition}\label{properties.16}
If a pre-Riesz space $X$ has the RDP, then $X$ is weakly pervasive.
\end{proposition}
\begin{proof}
Let $X$ be a pre-Riesz space with RDP. We use Lemma~\ref{properties.15}\ref{properties.15.eq1}. Let $b_1,b_2 \in X$ be such that $b_1, b_2>0$ and $b_1\not\perp b_2$, i.e. $i(b_1)\Inf i(b_2)> 0$ in a vector lattice cover $(Y,i)$ of $X$. We set $a_1:=0$. Since $X$ is order dense in $Y$, for the set
\[S:=\left\{z\in i(X)\mid| z\leq i(b_1)\Inf i(b_2)\right\}\]
we have $i(b_1)\Inf i(b_2) = \sup S$.
The element $a_1=0$ is a lower bound of $i(b_1)\Inf i(b_2)$. Assume that $S\leq 0$, then $\sup S\leq 0$ leads to a contradiction. Therefore there exists $a_2\in X$ with $a_2\not\leq 0$ and such that $i(a_2)\in S$. 
For the four elements we have the relationship $a_1, a_2 \leq b_1,b_2$. Since $X$ has the RDP, 
there exists $x\in X$ such that $a_1,a_2 \leq x \leq b_1,b_2$.
Since $a_1=0$ and $a_2\not\leq 0$, it follows $x>0$. 
We conclude that there exists $x>0$ such that $x\leq b_1,b_2$, i.e.\ $X$ is weakly pervasive.
\end{proof}

\begin{example}\label{properties.18}
\textit{An Archimedean weakly pervasive pre-Riesz space need not have the RDP.}\\
In \c[Example~23]{02} the following pre-Riesz space is considered:
\[X:=\t{PA}[-1,1]\oplus\left\{\lambda q \mid| \lambda \in\RR\right\}=\left\{ f+\lambda q \mid|f \in \t{PA}[-1,1], \lambda\in\RR  \right\},\]
where $q\in C[-1,1]$, $q(t) = t^2$ for $t\in[-1,1]$, and $\t{PA}[-1,1]$ is the set of all continuous piecewise affine functions on $[-1,1]$. It is established that $C[-1,1]$ is a vector lattice cover of $X$ with the identity $i$ as the embedding map, and that $X$ is pervasive. Hence weakly pervasive. Moreover, $X$ does not have the RDP. 
\end{example}
\begin{remark}
It is not clear whether weak pervasiveness is equivalent to the following property:
\begin{align*}\label{properties.15.eq3} 
(\t{P})\quad & \t{For every } b_1, \ldots, b_n\in X \t{ with } b_1, \ldots, b_n>0 \t{ and } i(b_1)\Inf\ldots\Inf i(b_n)\neq 0 \\
&\t{there exists an element } x\in X \t{ such that } 0 < x \leq b_1, \ldots, b_n.
\end{align*}
However, due to \c[Lemma~1.53]{AliTou} it is straightforward that RDP implies (P). On the other hand, the property (P) does not imply RDP. Indeed, in Example~\ref{properties.18} the space $X$ satisfies (P). Let $b_1,\ldots,b_n\in X$ with $b_1,\ldots,b_n> 0$ be such that $0<i(b_1)\Inf\ldots\Inf i(b_n)\in C[-1,1]$. Then there is a piecewise affine function $y\in C[-1,1]$ such that $0<y\leq i(b_1)\Inf\ldots\Inf i(b_n)$. Since $y$ is piecewise affine, we have for $x:=i(y)^{-1}\in X$ that $0<x\leq b_1,\ldots,b_n$. That is, $X$ satisfies (P).
\end{remark}

Note that a pervasive pre-Riesz space which does not have the RDP is given in \c[Example~23]{02}. The next example shows that RDP does not imply pervasiveness, in general.
\begin{example}\label{properties.7}
\textit{An Archimedean pre-Riesz space with RDP need not be pervasive.}

Consider the vector space of all (continuous) rational functions on $[0,1]$, i.e.\
\[X:=\left\{x\in C[0,1] \mid| x =\textstyle{\frac{p}{q}},\hs \forall t\in[0,1]: q(t)\neq 0, \t{ and } p,q \t{ polynomial}\right\},\]
endowed with the pointwise order.
Then $X$ is a subspace of the Archimedean vector lattice $C[0,1]$ and hence is Archimedean as well. Clearly, $X$ is directed and therefore pre-Riesz. Riesz established that $X$ has the RDP, see \c[Example 1.56]{AliTou}.

To show that $X$ is not pervasive, we use the characterization of Theorem~\ref{properties.2b}\ref{properties.2b.it3}, which avoids the computation of a vector lattice cover of $X$.
Let $b(t):=-16(t-\frac{1}{4})^2 +1$ for $t\in[0,1]$. Then the polynomial $b$ belongs to $X$. We have to show that there does not exist $x\in X$ such that $0<x \leq u$ holds for every $u\in X$ with $0,b\leq u$.

First consider for every $s\in\hs]\frac{1}{2},1]$ the positive parabola
\[u_s(t):=\frac{1}{(\frac{1}{2}-s)^2}(t-s)^2.\]
Notice that the maximal value of the parabola $b$ is the real number $b(\frac{1}{4})=1$, and for $t\in[\frac{1}{2},1]$ we have $b(t)\leq 0$. Moreover, for each $u_s$ with $s\in\hs]\frac{1}{2},1]$ we obtain $u_s(\frac{1}{2})=1$ and $u_s(s)=0$. It follows for the positive function $u_s$ that $b\leq u_s$ for every $s\in\hs]\frac{1}{2},1]$.

Assume now that there exists $x\in X$ such that $0<x\leq u$ holds for every positive $u$ with $b\leq u$. In particular, we have $0\leq x(s)\leq u_s(s) = 0$ for each $s\in\hs]\tfrac{1}{2},1]$. That is, $x(]\tfrac{1}{2},1])=0$.
However, the element $x$ has a representation $x = \frac{p}{q}$ with polynomials $p$ and $q$. Therefore the numerator $p$ must be zero on the whole of the interval $]\frac{1}{2},1]$. Due to $p$ being polynomial it follows $p=0$, which implies $x=0$, a contradiction to $x>0$. Theorem~\ref{properties.2b} yields that $X$ is not pervasive.
\end{example}

There are Archimedean pre-Riesz spaces, which are not vector lattices, but are weakly pervasive. A simple example is the space $C^1[0,1]$ of differentiable functions on the interval $[0,1]$.
On the other hand, not every Archimedean pre-Riesz space is weakly pervasive, as we see in the following example.
We use an example from \cite{Gaa2005}, where the focus is to construct an order bounded, non-regular linear functional on an Archimedean directed ordered vector space. The underlying space is a pre-Riesz space, but there is no obvious candidate for a vector lattice cover.
So far, it was an open question whether this pre-Riesz space is pervasive. Using the characterization in Lemma~\ref{properties.15}, we establish that the space is not even weakly pervasive.
\begin{example}\label{properties.14}
\textit{An Archimedean pre-Riesz space need not be weakly pervasive.}

As in \cite{Gaa2005}, for $A\subseteq [0,\infty[$ let $\1_{A}$ denote the corresponding indicator function, for $n,k\in\NN$ define
\begin{eqnarray*}
e_n: [0,\infty[\to \RR, & &t\mapsto\1_{[n-1,n[}(t),\\
u_{n,k}\colon [0,\infty[\to \RR, & &t\mapsto nt\1_{\left[0,\frac{1}{n}\right]}(t)+\textstyle\frac{1}{k}\1_{\left\{n+\textstyle\frac{1}{k}\right\}}(t),
\end{eqnarray*}
and consider the subspace
$X:=\operatorname{span}\left\{e_n, u_{n,k} \mid| n,k\in\NN\right\}$
of $\RR^{[0,\infty[}$ with point-wise order. 
As is established in \cite{Gaa2005}, $X$ is directed. Moreover, since $X$ is a subspace of the Archimedean vector lattice $\RR^{[0,\infty[}$, $X$ is Archimedean and therefore a pre-Riesz space.

We show that $X$ is not weakly pervasive. To that end we use the characterization in Lemma~\ref{properties.15}\ref{properties.15.eq1}.
For $b_1:=2u_{1,2}$ and $b_2:=e_2$ we have $b_1,b_2>0$. We show $b_1\not\perp b_2$ by establishing that $\left\{b_1-b_2,-b_1+b_2\right\}^u\neq \left\{b_1+b_2,-b_1-b_2\right\}^u$. Consider the element $v:=2u_{1,3}+e_2$. We show that $v\in \left\{b_1-b_2,-b_1+b_2\right\}^u$.
Indeed, for $t\in[0,1]$ we have $v(t) = 2u_{1,3}(t) = 2t \geq \pm 2t =\pm(b_1-b_2)(t)$. 
For $t\in[1,2]\ohne\left\{\tfrac{3}{2},\tfrac{4}{3}\right\}$ we obtain $v(t)=e_2(t) =1 \geq \pm 1 = \pm(b_1-b_2)(t)$.
Moreover, we get $v(\tfrac{4}{3})=\tfrac{5}{3} \geq \pm 1 = \pm(b_1-b_2)(\tfrac{4}{3})$ and $v(\tfrac{3}{2})=1 \geq \mp \tfrac{1}{2} = \pm(b_1-b_2)(\tfrac{3}{2})$.
We conclude $v\in \left\{b_1-b_2,-b_1+b_2\right\}^u$. On the other hand, $v\notin\left\{b_1+b_2,-b_1-b_2\right\}^u$, as $v(\tfrac{3}{2})=1 < 2=(b_1+b_2)(\tfrac{3}{2})$. This establishes $b_1\not\perp b_2$.

Assume there is $x\in X$ with $0< x \leq  b_1, b_2$. Then for $t\in [0,\infty[\ohne\left\{\tfrac{3}{2}\right\}$ we have $x(t)=0$ and $0< x(\tfrac{3}{2}) \leq 1$. It follows that $x$ is a non-zero multiple of $u_{1,2}$, which is a contradiction. That is, $X$ is not weakly pervasive.
\end{example}
\bibliographystyle{plain}

\begin{thebibliography}{10}

\bibitem{PosOp}
C.~D. Aliprantis and O.~Burkinshaw.
\newblock {\em Positive Operators}.
\newblock Springer, New York, 2006.

\bibitem{AliTou}
C.~D. Aliprantis and R.~Tourky.
\newblock {\em Cones and Duality}, volume~84 of {\em Graduate Studies in
  Mathematics}.
\newblock American Mathematical Society, Providence, Rhode Island, 2007.

\bibitem{113}
G.~Buskes and A.C.M. van Rooij.
\newblock The vector lattice cover of certain partially ordered groups.
\newblock {\em J. Austral. Math. Soc. (Series A)}, 54:352--367, 1993.

\bibitem{04}
A.~Kalauch and H.~Malinowski.
\newblock Order continuous operators on pre-{R}iesz spaces and embeddings.
\newblock \url{https://arxiv.org/abs/1802.02476}, 2018.
\newblock Submitted.

\bibitem{03}
A.~Kalauch and H.~Malinowski.
\newblock Vector lattice covers of ideals and bands in pre-{R}iesz spaces.
\newblock {\em Quaestiones Mathematicae}, 2018.
\newblock To appear. \url{https://doi.org/10.2989/16073606.2018.1501620}.

\bibitem{1}
A.~Kalauch and O.~van Gaans.
\newblock Disjointness in {P}artially {O}rdered {V}ector {S}paces.
\newblock {\em Positivity}, 10(3):573--589, 2006.

\bibitem{3}
A.~Kalauch and O.~van Gaans.
\newblock Bands in {P}ervasive pre-{R}iesz {S}paces.
\newblock {\em Operators and Matrices}, 2(2):177--191, 2008.

\bibitem{2}
A.~Kalauch and O.~van Gaans.
\newblock Ideals and {B}ands in {P}re-{R}iesz {S}paces.
\newblock {\em Positivity}, 12(4):591--611, 2008.

\bibitem{PRS}
A.~Kalauch and O.~van Gaans.
\newblock {\em Pre-Riesz spaces}.
\newblock De Gruyter, Berlin, 2018.

\bibitem{Zaa1}
W.~A.~J. Luxemburg and A.~C. Zaanen.
\newblock {\em {R}iesz {S}paces {I}}.
\newblock North-Holland Publ. Comp., Amsterdam, 1971.

\bibitem{02}
H.~Malinowski.
\newblock Order closed ideals in pre-{R}iesz spaces and their relationship to
  bands.
\newblock {\em Positivity}, 22(4):1039–1063, 2018.
\newblock \url{https://arxiv.org/abs/1801.09963}.

\bibitem{Gaa2005}
O.~van Gaans.
\newblock An elementary example of an order bounded dual space that is not
  directed.
\newblock {\em Positivity}, 9(2):265--267, 2005.

\bibitem{vanHaa}
M.~van Haandel.
\newblock {\em Completions in {R}iesz Space Theory}.
\newblock Proefschrift (PhD thesis), Katholieke Universiteit Nijmegen,
  Nederland, 1993.

\bibitem{vanImhoff2017}
H.~van Imhoff.
\newblock Riesz* homomorphisms on pre-{R}iesz spaces consisting of continuous
  functions.
\newblock {\em Positivity}, 22(2):425–447, 2018.

\bibitem{Waaij}
J.~van Waaij.
\newblock Tensor products in {R}iesz space theory.
\newblock Master's thesis, Mathematisch Instituut, Universiteit Leiden,
  Nederland, 2013.

\end{thebibliography}

\end{document}